\documentclass[graybox]{svmult}

\pdfpageattr{
  /CropBox [100 100 500 740]
}

\usepackage{amsmath,amssymb,amsfonts}
\usepackage{graphicx,epsfig,url}
\usepackage{psfrag}
\usepackage{mathptmx}       
\usepackage{helvet}         
\usepackage{courier}        
\usepackage{makeidx}         
\usepackage[bottom]{footmisc}
\usepackage{float}
\usepackage{color,xcolor}

\usepackage{latexsym}
\usepackage{pstricks}
\usepackage{bbm}
\definecolor{lightblue}{rgb}{0.22,0.45,0.70}
\usepackage[colorlinks=true,breaklinks=true,linkcolor=lightblue,citecolor=lightblue]{hyperref}
\usepackage{cleveref}

\newcommand{\cblue}[1]{{\leavevmode\color{blue}{#1}}}

\newcommand{\dx}{\,\mbox{d}x}

\newcommand{\CP}{\mathcal{P}}

\newcommand{\CT}{\mathcal{T}}
\newcommand{\CV}{\mathcal{V}}
\newcommand{\CJ}{\mathcal{J}}

\newcommand{\bF}{\mbox{\boldmath{$F$}}}
\newcommand{\bP}{\mbox{\boldmath{$P$}}}

\newcommand{\bn}{\mbox{\boldmath{$n$}}}

\newcommand{\bx}{\mbox{\boldmath{$x$}}}

\newcommand{\brho}{\mbox{\boldmath{$\rho$}}}
\newcommand{\bsigma}{\mbox{\boldmath{$\sigma$}}}
\newcommand{\sbsigma}{\mbox{\boldmath{\scriptsize $\sigma$}}}

\newcommand{\bphi}{\mbox{\boldmath{$\phi$}}}
\newcommand{\bpsi}{\mbox{\boldmath{$\psi$}}}

\newcommand{\btau}{\mbox{\boldmath{$\tau$}}}
\newcommand{\sbtau}{\mbox{\boldmath{\scriptsize $\tau$}}}

\newcommand{\bM}{\mathbf{M}}
\newcommand{\bbV}{\mathbb{V}}

\newcommand{\ol}{\overline}

\newcommand{\BBR}{\mbox{$\mathbb{R}$}}

\newfont{\twelvemsb}{msbm10 at 11.6pt}

\renewcommand{\div}{\mathop{\rm div}\nolimits}

\DeclareMathOperator*{\argmin}{arg\,min}

\usepackage{colordvi}

\allowdisplaybreaks

\makeindex

\begin{document}
\title*{An ultra-weak three-field finite element  formulation for the biharmonic and extended Fisher--Kolmogorov equations}
\titlerunning{Ultra-weak mixed FEM for biharmonic and extended Fisher--Kolmogorov problems}
\author{Rekha Khot, Bishnu P. Lamichhane, and Ricardo Ruiz-Baier}
\authorrunning{Khot, Lamichhane \& Ruiz-Baier}
\institute{Rekha Khot \at Department of Mathematics, Indian Institute of Technology Palakkad, Kanjikode, Palakkad, Kerala
678623, India, \email{RekhaKhot@iitpkd.ac.in}
\and Bishnu P.~Lamichhane \at School of Mathematical \& Physical Sciences,
University of Newcastle,
University Drive,
Callaghan, NSW 2308, Australia,  \email{Bishnu.Lamichhane@newcastle.edu.au}
\and Ricardo Ruiz-Baier \at School of Mathematics, Monash University, 9 Rainforest Walk, Melbourne 3800 VIC, Australia, 
and Universidad Adventista de Chile, Casilla 7-D, Chill\'an, Chile, \email{Ricardo.RuizBaier@monash.edu}} 

\maketitle
\abstract{This paper discusses a so-called ultra-weak three-field formulation of the biharmonic problem where the solution, its gradient, and an additional Lagrange multiplier are the three unknowns. We establish the well-posedness of the problem using the abstract theory for saddle-point problems, and develop a conforming finite element scheme based on Raviart--Thomas discretisations of the two auxiliary variables. The well-posedness of the discrete formulation and the corresponding a priori error estimate are proved using a discrete inf-sup condition. 
We further extend the analysis to the  time-dependent semilinear equation, namely extended Fisher--Kolmogorov equation. 
We present a few numerical examples to demonstrate the performance of our approach.}



\section{Introduction}

Fourth-order partial differential equations appear in many 
applications. Some examples are thin beams and plates, strain 
gradient elasticity \cite{Cia78,EGH02}, the Stokes problem 
\cite{GR86} and  phase separation of a binary mixture \cite{WKG06}.
The variational formulation of these problems requires 
$H^2$-conforming finite elements, which are not so easy to construct 
on unstructured meshes. There are a few classical approaches to avoid this difficulty.  The first one is to use  discontinuous Galerkin methods or non-conforming approaches as in \cite{Morley, EGH02,BS05,WKG06}. The non-conforming approach based on Morley finite element is quite efficient but it does not produce a \textit{uniform} approximation when applied to a singularly perturbed problem \cite{MTW02}.
The second approach is to apply a mixed formulation as in \cite{CR74,CG75,Fal78,Cia78,FO80,BOP80,Mon87,DP01,farrell2022new}.  

This paper is concerned about a mixed formulation of 
the biharmonic equation and its finite element discretisation. 
As in  \cite{JP82,CHH00,Lam11b, BLS17}, 
our formulation is  based on using the gradient of the solution as  a new variable.
We formulate our problem as 
minimising the square of the $L^2$-norm of 
the divergence of the gradient  in a suitable Hilbert space as in \cite{BLS17}.
Then we write a variational equation of the constraint using a Lagrange multiplier. 
However,  in contrast to  \cite{BLS17}, we reformulate the constraint using  an integration by parts to recast the problem with the solution in the $L^2$-space.
Therefore, we call our formulation \emph{ultra-weak}. 
The gradient of the solution and the Lagrange multiplier both are discretised using 
a Raviart--Thomas finite element space. We note that a very similar formulation has been introduced in \cite{farrell2022new} for the biharmonic problem. Here we also carry out the error analysis for the extended Fisher--Kolmogorov (EFK) equation, which is a time-dependent {semi}linear fourth-order equation with a parameter that  multiplies the biharmonic term and thus controls the strength of higher-order diffusion/curvature penalty.

The contents of the paper have been organised as follows. Section~\ref{sec:mixed} describes the strong form of the biharmonic problem with different boundary conditions and shows a derivation of the proposed ultra-weak formulation.  In Section~\ref{sec:wellp}, we address the unique solvability of the problem using the Babu\v{s}ka--Brezzi theory for saddle-point problems. We introduce the discrete problem in Section~\ref{sec:fe}, show discrete well-posedness, and establish quasi-optimality results, which for Raviart--Thomas finite elements imply optimal convergence rates. Then, we 
extend the approach to the extended Fisher--Kolmogorov equation in Section~\ref{sec:EFK},  prove a stability estimate and show the convergence of the numerical method.
Finally, Section~\ref{sec:results} presents numerical examples in 2D and 3D illustrating the properties of the proposed three-field formulation, and focusing on the extended Fisher--Kolmogorov problem. We also present a brief parametric study with the performance of the mixed finite element scheme when varying the curvature coefficient. 

\section{A mixed formulation for the biharmonic equation}\label{sec:mixed}
In this section we discuss a new mixed formulation of the 
biharmonic problem. Let  $\Omega \subset \BBR^d$, $d \in \{2,3\}$,  be a bounded convex domain 
with polygonal or polyhedral boundary $\Gamma=\partial\Omega$ and 
 an unit normal $\bn$  pointing outward to $\Gamma$.
We consider the biharmonic equation 
\begin{equation}\label{biharm}
\Delta^2 u =f \quad\text{in}\quad \Omega
\end{equation}
 with either simply supported boundary conditions
 \begin{equation}\label{bc1}
u= \Delta u =0\quad\text{on}\quad \Gamma,
\end{equation}
or with Cahn--Hilliard boundary conditions
\begin{equation}\label{bc2}
\frac{\partial u}{\partial \bn}=\frac{\partial \Delta u}{\partial \bn}=0\quad\text{on}\quad \Gamma.
\end{equation}

Here we consider a new mixed formulation of 
the biharmonic equation and its finite element discretisation. We aim to derive an  ultra-weak formulation, where 
the solution $u$ is only square integrable.  
There are many  mixed schemes for the biharmonic 
problem \cite{CR74,CG75,Cia78,JP82,CHH00,Lam11a,Lam11b}. 
Our formulation is similar to a formulation proposed in 
a recent contribution \cite{farrell2022new}, 
where the solution $u$ belongs to $L^2(\Omega)$. 
In this article, we derive a unified formulation for simply supported \eqref{bc1} and Cahn--Hilliard \eqref{bc2} boundary conditions (BCs), and extend the analysis to the extended Fisher--Kolmogorov equation. 

Here we start with a minimisation problem for the biharmonic 
equation and introduce the gradient as a new unknown as in 
\cite{CR74,JP82,CHH00,Lam11a,Lam11b}. Then we write 
a variational equation of this new equation using a Lagrange multiplier. 

We use the Sobolev spaces $L^2(\Omega)$,  $H^{s}(\Omega)$ for $s \in \BBR$ 
and $H(\div,\Omega)$, which  are defined in a standard way, equipped with their usual inner products and norms \cite{LM72,Ada75,Cia78,BS94}.  
Defining
\begin{align*}
 & V:= \begin{cases}
 H^2(\Omega)\cap H_0^1(\Omega) &  \text{ for simply supported BCs,}\\
 \{ v\in H^2(\Omega): \,\, \partial v /\partial \bn =0 \,\, \text{on}\,\, \Gamma, \,\, \int_{\Omega}v=0 \} & \text{ for Cahn--Hilliard BCs,}
\end{cases}
\end{align*}
we consider the following variational form of the biharmonic problem \eqref{biharm}
\begin{equation}\label{vbiharm}
J(u)=\inf_{v \in V}J(v),
\end{equation}
with 
\begin{equation}\label{func}
J(v):=\frac{1}{2} \int_{\Omega}|\Delta v|^2 \dx-\ell(v),\quad \ell(v):=\int_{\Omega}f\,v  \dx.
\end{equation}
The minimisation problem associated with \eqref{func} is equivalent to finding $u \in V$
such that 
\begin{equation}\label{weak}
 A(u,v) = \ell(v) \qquad \forall \; v \in V,
 \end{equation}
where 
\[A (u, v) := \int_{\Omega} \Delta u \Delta v \dx.\]
The existence and uniqueness of the solution of \eqref{vbiharm} and \eqref{weak}
are well established \cite{CR74,Cia78,DP01}.

We need the following two Hilbert spaces to define our new mixed formulation. 
Let \[ W := \begin{cases} H^1_0(\Omega) \quad&\text{for simply supported BCs}, \\
\{ v \in H^1(\Omega):\, \int_{\Omega}v \dx =0\}\quad&\text{for Cahn-Hilliard BCs},
\end{cases}
\]
and 
\[ \bM := \begin{cases} H(\div, \Omega) \quad&\text{for simply supported BCs},\\
H_0(\div,\Omega)\quad&\text{for Cahn-Hilliard BCs},
\end{cases}
\]
where 
\[ H_0(\div,\Omega) := \{ \btau \in H(\div,\Omega):\, 
\btau \cdot \bn = 0 \;\text{on}\; \Gamma\}.\]

Introducing the new variable $\bsigma := \nabla u$ (as in the Ciarlet--Raviart formulation, $\bsigma$ represents the velocity of the flow in the context of Stokes equations, or the rotation vector in the case of plate equations), the biharmonic problem is now recast as a constrained minimisation problem \cite{BLS17, Cia78, CR74}
\begin{equation*}\label{mspline} 
\argmin\limits_{\substack{(u,\sbsigma) \in [W \times \bM]\\ \sbsigma=\nabla u}} \left(
\frac{1}{2}\|\div  \bsigma\|_{L^2(\Omega)}^2 -\ell(u)\right).
\end{equation*}
\begin{remark}
The spaces $W$ and $\bM$ are defined in 
such a way that we get the correct boundary conditions from this minimisation formulation 
for both types of boundary conditions.
\end{remark}

To arrive at an ultra-weak formulation, we multiply the constraint 
$ \bsigma=\nabla u $ by a Lagrange multiplier  $\bpsi \in \bM $ and 
integrate by parts to get 
\[ \int_{\Omega} \bsigma  \cdot \bpsi\dx+ \int_{\Omega} u \div  \bpsi\dx =0.\]
Let \[ U := \begin{cases} L^2(\Omega) \quad&\text{for simply supported BCs},\\
\{ v \in L^2(\Omega):\, \int_{\Omega}v \dx =0\}\
\quad&\text{for Cahn--Hilliard BCs}.
\end{cases}
\]
Let $\bbV := U \times   \bM$. Here $\bbV$ is equipped with the 
standard graph norm: 
\[ \|(v,\btau)\|^2_{\bbV} := \|v\|_{0,\Omega}^2+\| \btau\|_{\div,\Omega}^2,\]
where $\|\btau\|_{\div,\Omega}^2:=\|\btau\|^2_{0,\Omega}+\|\div \btau\|^2_{0,\Omega}.$

We now define the constrained space $\CV$ as 
\[ \CV :=\{ (v,\btau) \in \bbV:\, \int_{\Omega} \btau  \cdot \bpsi \dx+ \int_{\Omega} v \div  \bpsi \dx =0, \quad \forall\bpsi \in \bM\}.\]
This leads to a minimisation problem of finding $(u,\bsigma) \in \bbV$ such that 
\begin{equation}\label{vbnew}
\CJ(u,\bsigma)=\inf_{(v,\sbtau) \in \CV}J(v,\btau),
\end{equation}
with 
\begin{equation*}
J(v,\btau):=\frac{1}{2} \int_{\Omega}|\div  \btau |^2 \dx-\ell(v).
\end{equation*}

\begin{remark}
A similar three-field formulation has been proposed in \cite{BLS17}, where the constraint is 
written as 
\[  \int_{\Omega} (\bsigma -\nabla u) \cdot \bpsi \dx=0,\qquad 
\forall\bpsi \in [L^2(\Omega)]^d, \]
and  $u \in H^1_0(\Omega)$.
In that case, a discrete Lagrange multiplier space should be chosen very carefully 
to satisfy the stability. Invoking an integration by parts and using the boundary condition, 
our new formulation requires the solution $u \in L^2(\Omega)$. 
Therefore, we name it an ultra-weak formulation.
\end{remark}

 \begin{remark}
If $u \in V$ is the solution to the minimisation problem \eqref{vbiharm}, then 
$(u,\nabla u)\in \CV$ is also the solution to the minimisation problem \eqref{vbnew}.
Conversely, if $(u,\bsigma)\in \CV$ is the solution to the minimisation problem \eqref{vbnew}, we get  \[ 
\int_{\Omega} \bsigma  \cdot \bpsi \dx+ \int_{\Omega} u\div  \bpsi \dx =0, \qquad \forall\bpsi \in \bM.
\] 
If $ u \in H^2(\Omega)$, an integration by parts leads to
\[ 
\int_{\Omega} \bsigma  \cdot \bpsi \dx- \int_{\Omega} \nabla u \cdot \bpsi \dx + 
\int_{\Gamma} u \bpsi\cdot \bn\,d\sigma=0, \qquad \forall\bpsi \in \bM.
\] 
In particular, choosing $\bpsi\in \bM$ with $\bpsi\cdot\bn=0$ on $\Gamma$ produces
$\bsigma = \nabla u$ for both cases of boundary conditions. Using 
$\bsigma = \nabla u$ and the fact that $\bpsi \in \bM$, we get 
$ u =0$ on $\Gamma$ for the case of simply supported boundary conditions. For the case of 
Cahn-Hilliard boundary conditions, as $\bsigma \in \bM$, we have 
$ \bsigma \cdot \bn = 0$ on $\Gamma$, which gives 
$\nabla u \cdot \bn =0$ on $\Gamma$.  Hence $u \in V$ for both types of boundary conditions, and  $u$ is the solution to the minimisation problem \eqref{vbiharm}. 
\end{remark}

This leads to a saddle point problem of finding 
$ ((u,\bsigma),\bphi) \in \bbV \times \bM$ such that 
\begin{subequations}\label{saddle}
    \begin{alignat}{3}
     &a((u,\bsigma),(v,\btau)) +  b((v,\btau),\bphi) & {}={} & \ell(v),\qquad  &\forall&(v,\btau) \in \bbV, \\
     & b((u,\bsigma),\bpsi)& ={} & 0\qquad
  &\forall&\bpsi \in  \bM,
    \end{alignat}
\end{subequations}
where  bilinear forms $a(\cdot,\cdot)$ and 
$b(\cdot,\cdot)$  are given by
\begin{subequations}\label{def}
\begin{align}
a((u,\bsigma),(v,\btau))&:= 
\int_{\Omega} \div   \bsigma \,\div  \btau \dx,\label{def:a}\\
b((v,\btau),\bpsi)&:=\int_{\Omega} \btau\cdot \bpsi + v\div  \bpsi \dx.\label{def:b}
\end{align}
\end{subequations}
\begin{remark}
If $u \in H^1(\Omega)$, the second equation of \eqref{saddle} yields $\bsigma = \nabla u$.
Using $v=0$ in  the first equation of \eqref{saddle} and integrating by parts, we see that 
$\bphi = \nabla (\Delta u)$ if $u \in H^3(\Omega)$.
\end{remark}

\section{Stability analysis}\label{sec:wellp}
In this section, we show the coercivity of the bilinear form $a(\cdot,\cdot)$ and inf-sup condition of the bilinear form $b(\cdot,\cdot)$ on the kernel space $\CV$.  
\begin{lemma}
There exists a constant $C>0$ such that for $(v, \btau) \in \CV$ 
\[ a((v,\btau),(v,\btau)) \geq C \left(\|\btau \|_{\div,\Omega}^2+ \|v\|_{0,\Omega}^2\right).\]
\end{lemma}
\begin{proof}
Let $(v, \btau) \in \CV$. Notice from the definition \eqref{def:a} of $a(\cdot,\cdot)$ that
\begin{equation}\label{eq0}
a((v,\btau),(v,\btau)) = \|\div  \btau\|^2_{0,\Omega}.\end{equation}
We aim to show that
\[ \|v\|_{0,\Omega} \leq C \|\btau\|_{0,\Omega},\quad\text{and}\quad \|\btau\|_{0,\Omega} \leq C \|\div \btau\|_{0,\Omega}.
\]
Since $(v, \btau) \in \CV$, we have 
\begin{equation}\label{eq1}
\int_{\Omega} \btau\cdot \bpsi + v\div  \bpsi \dx=0,
\quad \forall\bpsi \in \bM.
\end{equation}
We choose $\bpsi\in \bM$ as a solution of the divergence  equation \cite[Lemma 11.2.3]{BS94}
\[ \div \bpsi = v\] with the property 
\[ \|\bpsi\|_{0,\Omega} \leq C \|v\|_{0,\Omega}.
\]
This in \eqref{eq1} and  the Cauchy--Schwarz inequality lead to
\[\|v\|_{0,\Omega}^2  = -\int_{\Omega} \btau\cdot \bpsi \dx 
\leq \|\btau\|_{0,\Omega}\|\bpsi\|_{0,\Omega}\leq C
\|\btau\|_{0,\Omega}\|v\|_{0,\Omega}. \]
Thus, we have 
\begin{equation}\label{eq3} \|v\|_{0,\Omega} \leq C \|\btau\|_{0,\Omega}.
\end{equation}
For the second step, choosing $\bpsi = \btau$ in \eqref{eq1}, we
obtain 
\begin{equation*}
\|\btau\|_{0,\Omega}^2 = -\int_{\Omega} v\div  \btau \dx.
\end{equation*}
This followed by the Cauchy--Schwarz inequality and \eqref{eq3} proves that 
\begin{equation}\label{eq4}
\|\btau\|_{0,\Omega} \leq C \|\div  \btau\|_{0,\Omega}.
\end{equation}
The conclusion follows by using 
\eqref{eq0}, \eqref{eq3} and \eqref{eq4}.
\end{proof}
\begin{lemma}
The bilinear form $b(\cdot,\cdot)$ satisfies the inf-sup condition:
\begin{eqnarray}\label{infsupc} 
\sup_{(v,\sbtau) \in \bbV }
\frac{b((v,\btau),\bpsi)}
{\|(v,\btau)\|_{\bbV}} 
\geq  \|\bpsi\|_{\div, \Omega}, 
\qquad \forall\bpsi\in \bM.
\end{eqnarray}
Hereafter, the supremum is taken 
over all functions of the underlying space with 
non-zero norm on the denominator.
\end{lemma}
\begin{proof}
Let $\bpsi\in \bM$. Choosing $v=\div \bpsi$ and $\btau = \bpsi$, we have
\begin{align*}
\sup_{(v,\sbtau) \in \bbV }
\frac{b((v,\btau),\bpsi)}
{\|(v,\btau)\|_{\bbV}} \geq \sqrt{\|\bpsi\|^2_{0,\Omega}+\|\bpsi\|_{\div, \Omega}^2}\geq \|\bpsi\|_{\div, \Omega}.
\end{align*}
This concludes the proof.

\end{proof}

\begin{remark}
The biharmonic equation 
with clamped boundary conditions $
u =\frac{\partial u}{\partial \bn}=0$ on $\Gamma$
corresponds to the  saddle point problem of finding 
$ ((u,\bsigma),\bphi) \in \bbV \times H(\div, \Omega) $  with $\bbV = L^2(\Omega) \times H_0(\div,\Omega)$ such that 
\begin{alignat*}{3}
 &a((u,\bsigma),(v,\btau)) + b((v,\btau),\bphi) &{}={}& \ell(v),\qquad&\forall&
(v,\btau) \in \bbV, \\
&b((u,\bsigma),\bpsi) & {} = {} & 0, \qquad&\forall&  \bpsi \in  H(\div, \Omega).
\end{alignat*}
In this case, the formulation is not well-posed as the inf-sup condition \eqref{infsupc} does not hold. 
A similar issue exists for the well-known Ciarlet--Raviart formulation \cite{CR74,BOP80}.
 
\end{remark}

\section{Finite element discretisation} \label{sec:fe}
Let $\CT_h$ be a shape-regular partition of the domain  $\Omega$ in 
simplices having the mesh-size $h$. 
Note that for $d=2$ a simplex is a triangle, and for $d=3$ 
it is a tetrahedron. In the following, we use a positive generic constant 
$C$, which may take different values at different places but is always 
independent of the mesh-size $h$.  
Let $\CP_k(T)$ be the space of polynomials of degree $k \geq 0$ in the simplex $T$.
We consider the following finite dimensional subspaces $U_h$ and $\bM_h$ 
based on the triangulation $\CT_h$ of 
the Hilbert spaces $U$ and $\bM$, respectively  as follows: 
\begin{eqnarray*}
 U_h &:=&
\{ u_h \in U: \, u_h|_{T}  \in \CP_k(T), \quad \forall T \in \CT_h\}, \\
\bM_h&:=&\{\bphi_h \in \bM:\,\bphi_h|_{T} \in  
\left[ \CP_k(T) \right]^d + \bx \, \CP_k(T),\quad \forall T \in \CT_h\},
\end{eqnarray*}
where $\bx\in \BBR^d$ is the space coordinate vector.
Here $\bM_h$ is the Raviart--Thomas space of order \( k \) \cite{Bra01,BBF13}. 
Let $\bbV_h := U_h \times \bM_h$.
The finite element problem is to find 
$((u_h,\bsigma_h),\bphi_h) \in \bbV_h \times \bM_h$ such that 
\begin{subequations}\label{dsaddle}
    \begin{alignat}{3}
     &a((u_h,\bsigma_h),(v_h,\btau_h)) + b((v_h,\btau_h),\bphi_h) &{}={}&   \ell(v_h),\quad 
&\forall& (v_h,\btau_h) \in \bbV_h, \\
&b((u_h,\bsigma_h),\bpsi_h) &{}={}& 0\quad &\forall& \bpsi_h \in  \bM_h.
    \end{alignat}
\end{subequations}
Now we prove the   well-posedness of the discrete formulation as in the continuous setting. The continuity of both bilinear forms $a(\cdot,\cdot)$ and $b(\cdot,\cdot)$, 
and the linear form $\ell(\cdot)$ follows 
as in the continuous setting as we use conforming spaces. 
To prove the inf-sup condition for the bilinear form $b(\cdot,\cdot)$, we choose $v_h=\div \bpsi_h $ and $\btau_h = \bpsi_h$ for all $\bpsi_h \in \bM_h$ as in the continuous setting to get 
 \begin{eqnarray*}
\sup_{(v_h,\sbtau_h) \in U_h \times  \bM_h }
\frac{b((v_h,\btau_h),\bpsi_h)}{ \|(v_h,\btau_h)\|_{\bbV} }=
\sup_{(v_h,\sbtau_h) \in U_h \times  \bM_h }
\frac{b((v_h,\btau_h),\bpsi_h)}{\sqrt{\|v_h\|^2_{0,\Omega}+ \|\btau_h\|^2_{\div,\Omega}}}
\geq  \|\bpsi_h\|_{\div,\Omega}.
\end{eqnarray*}

Now we turn our attention to prove the  coercivity of the bilinear form $a(\cdot,\cdot)$ on 
 the kernel space $\CV_h$ defined by 
\[ \CV_h :=\{ (v_h, \btau_h) \in U_h  \times \bM_h:\,
b((v_h, \btau_h),\bpsi_h) =  0,\quad \forall\bpsi_h \in \bM_h\}.\]
Let $(v_h, \btau_h) \in \CV_h$ so that 
 \begin{equation}\label{eqn2}
 \int_{\Omega} \btau_h\cdot \bpsi_h \dx + \int_{\Omega} 
 v_h \div  \bpsi_h \dx =0, \quad \forall\bpsi_h \in \bM_h. 
 \end{equation}
We choose $\bpsi_h\in \bM_h$ as a solution of the divergence  equation (see, e.g., \cite{BBF13}) 
\[ \div \bpsi_h = v_h\] with the property 
\[ \|\bpsi_h\|_{0,\Omega} \leq C \|v_h\|_{0,\Omega},
\]
in \eqref{eqn2} and  get 
\[\|v_h\|_{0,\Omega}^2  = -\int_{\Omega} \btau_h\cdot \bpsi_h \dx 
\leq \|\btau_h\|_{0,\Omega}\|\bpsi_h\|_{0,\Omega}\leq C
\|\btau_h\|_{0,\Omega}\|v_h\|_{0,\Omega}. \]
This gives us 
\begin{equation}\label{eqn4}
\|v_h\|_{0,\Omega} \leq C \|\btau_h\|_{0,\Omega}. 
\end{equation}
Choosing $\bpsi_h= \btau_h$, we have 
 \[\|\btau_h\|^2_{0,\Omega}  =  - \int_{\Omega} v_h \div \btau_h \dx.
 \]
 This yields 
 \begin{equation}\label{eqn3}
\|\btau_h\|^2_{0,\Omega}   \leq \|v_h\|_{0,\Omega} \|\div  \btau_h\|_{0,\Omega}.
 \end{equation}
Combining \eqref{eqn4} with \eqref{eqn3}, we get 
\begin{equation}\label{eqn5}
\|\btau_h\|_{0,\Omega}   
\leq C \|\div  \btau_h\|_{0,\Omega}.
\end{equation}
Hence, from \eqref{eqn5} we have the coercivity result 
\[ 
 a((v_h,\btau_h),(v_h,\btau_h))  = 
 \| \div  \btau_h \|_{0,\Omega}^2 \geq C \|\btau_h \|_{\div,\Omega}^2 + 
 \|v_h\|_{0,\Omega}^2, \quad \forall(v_h,\btau_h) \in \CV_h.
 \]
 
The following theorem for the a priori error estimate follows from the theory of 
mixed finite elements \cite{BF91}.
\begin{theorem}\label{th:cv}
There exists a unique solution $((u_h,\bsigma_h),\bphi_h) \in \bbV_h \times \bM_h$ to the discrete saddle-point problem \eqref{dsaddle}. Moreover, we have that  
\[ \|u_h\|_{0,\Omega} + \|\bsigma_h\|_{\div,\Omega} + \|\bphi_h\|_{\div,\Omega} \leq C \|f\|_{0,\Omega},\]
and for the solution $((u, \bsigma), \bpsi)\in \bbV \times \bM$ of the saddle point 
problem \eqref{saddle}, we have 
\begin{eqnarray*}
&&\|u-u_h\|_{0,\Omega} + \|\bsigma-\bsigma_h\|_{\div,\Omega}  +  \|\bphi-\bphi_h\|_{\div,\Omega}  \\
&& \leq  
  C \biggl(\inf_{v_h \in  U_h} \|u-v_h\|_{0,\Omega} + \inf_{\btau_h \in  \bM_h}   \|\bsigma-\btau_h\|_{\div,\Omega}
 +\inf_{\bpsi_h \in  \bM_h}   \|\bphi-\bpsi_h\|_{\div,\Omega}\biggr).
 \end{eqnarray*}
 Moreover, if $u \in H^{k+1}(\Omega)$, $\bsigma,\bphi \in [H^{k+1}(\Omega)]^d $ and 
 $\div \bsigma,\div\bphi \in H^{k+1}(\Omega) $, we have 
\begin{eqnarray*}
&&\|u-u_h\|_{0,\Omega} + \|\bsigma-\bsigma_h\|_{\div,\Omega}  +  \|\bphi-\bphi_h\|_{\div,\Omega}  \\
&& \leq  
  C h^{k+1} \left(\|u\|_{k+1,\Omega} + \|\bsigma\|_{k+1,\Omega} + \|\div\bsigma\|_{k+1,\Omega} + 
  \|\bphi\|_{k+1,\Omega} + \|\div \bphi\|_{k+1,\Omega}\right).
 \end{eqnarray*}
 
\end{theorem}
\section{Extended Fisher--Kolmogorov equation}\label{sec:EFK}
In this section, we apply our mixed formulation to the following time-dependent fourth-order extended Fisher--Kolmogorov equation
\begin{subequations}\label{def:time_dep_model_problem}
\begin{alignat}{2}
 \partial_t u+ \gamma \Delta^2 u -\Delta u + g(u)&=0 &\qquad&\text{in}\,\,J\times\Omega, \\
 u(x,0)&=u_0(x) &\qquad&\text{in}\,\,\Omega,
\end{alignat}
\end{subequations}
where the time interval is $J:=(0,T]$ and the nonlinear term is of the form $g(u):= u^3-u$. The positive parameter $\gamma$ controls the energetic penalty for curvature (bending) of $u$, acting as a higher-order diffusion that suppresses small-scale structure and introduces an intrinsic length scale (see, e.g., \cite{dee1988bistable, peletier1997spatial}). Applying the mixed formulation for the biharmonic problem introduced earlier, we get 
the following time-dependent saddle point problem: For every $t>0$, find  
$ ((u(t),\bsigma(t)),\bphi(t)) \in \bbV \times \bM$ such that 
\begin{subequations}\label{saddle-cont}
    \begin{alignat}{2}
    (\partial_tu(t), v)_{0,\Omega} +\gamma \left( a((u(t),\bsigma(t)),(v,\btau)) +  b((v,\btau),\bphi(t))\right) & \nonumber\\
      + c(\bsigma(t),\btau)+(g(u(t)),v)_{0,\Omega}&={}  \ell(v)\quad  &\forall&(v,\btau) \in \bbV, \\
     b((u(t),\bsigma(t)),\bpsi)& ={}  0\quad
  &\forall&\bpsi \in  \bM
    \end{alignat}
\end{subequations}
with $u(.,0) =u_0(x)$. From now on, and whenever clear from the context, we will drop the notation $(t)$ for the trial functions. Recall the bilinear forms $a(\cdot,\cdot)$ and 
$b(\cdot,\cdot)$  from \eqref{def}. We denote the usual $L^2$-inner product by $(\cdot,\cdot)_{0,\Omega}$, and define the bilinear form 
\begin{align*}
c(\bsigma,\btau)&:= 
\int_{\Omega}   \bsigma \cdot \btau \,\dx.
\end{align*}
Since $g(u)=u^3-u$ is a locally Lipschitz continuous function, there exists a unique solution to \eqref{saddle-cont} for a finite time $T>0$ \cite{Pazy}.

We use the following shorthand notation:
\begin{align*} \|\bullet\|_{L^p(J;*)}^{p}:=\int_{J}\|\bullet(s)\|_{*}^p\,ds\quad\text{for}\;p\in[1,\infty),\quad \|\bullet\|_{L^\infty(\ol{J};*)}:=\sup_{s\in\ol{J}}\|\bullet(s)\|_{*},
\end{align*}
where the (seminorm) $\|\cdot\|_{*}$  depends on the context.
\begin{lemma}
The following stability estimate holds:
\begin{align*}
\frac12\|u\|^2_{L^\infty(\ol{J};L^2(\Omega))} +\gamma \|\div \bsigma\|^2_{L^2(J);L^2(\Omega))}+\|\bsigma\|^2_{L^2(J;L^2(\Omega))} \lesssim \|u_0\|^2_{0,\Omega}+\|f\|^2_{L^2(J;L^2(\Omega))}.
\end{align*}
\end{lemma}
\begin{proof}
Choosing $v=u, \btau=\bsigma$ and $\bpsi=\bphi$ and using $b((u,\bsigma),\bphi)=0$, we get   
\begin{align*}
\frac12\frac{d}{dt}\|u\|^2_{0,\Omega} + \gamma \|\div \bsigma\|^2_{0,\Omega} +\|\bsigma\|^2_{0,\Omega} &= \ell(u)-(g(u),u)_{0,\Omega}\\
& = (f,u)_{0,\Omega} -\|u^2\|^2_{0,\Omega} + \|u\|^2_{0,\Omega}\\
&\leq \frac12\|f\|^2_{0,\Omega} + \frac{3}2\|u\|^2_{0,\Omega},
\end{align*}
with the definition of $g(u)=u^3-u$ in the second step and the Cauchy--Schwarz inequality followed by the Young inequality in the last step. Invoking Gronwall's lemma, we conclude the proof.
\end{proof}
This stabiliy result and the local existence implies the global existence of the unique solution to \eqref{saddle-cont}.
   The corresponding finite element problem is to find, for all $t>0$, 
$((u_h(t),\bsigma_h(t)),\bphi_h(t)) \in \bbV_h \times \bM_h$ such that 
\begin{subequations}\label{saddle-dis}
    \begin{alignat}{2}
     (\partial_tu_h, v_h)_{0,\Omega} +\gamma (a((u_h,\bsigma_h),(v_h,\btau_h)) +  b((v_h,\btau_h),\bphi_h)) \\
  \nonumber   + c(\bsigma_h,\btau_h)+(g(u_h),v_h)_{0,\Omega}&=  \ell(v_h)&\quad&\forall (v_h,\btau_h) \in \bbV_h, \\
      b((u_h,\bsigma_h),\bpsi_h) &=  0& &\forall \bpsi_h \in  \bM_h
    \end{alignat}
\end{subequations}
with $u_h(\cdot,0) = P_h u_0(x)$, where $P_h$ is the standard $L^2$-projection, that is, for all $v_h\in U_h$,
\begin{align}
(v-P_hv,v_h)_{0,\Omega} &= 0 \qquad\forall v_h\in U_h.\label{L2}
\end{align}
The existence and uniqueness of the solution $u_h$ follows analogously as in the continuous problem. 

Let us recall that if $v\in H^k(\Omega)$, then the following approximation property holds:
\begin{align*}
\|v-P_hv\|_{0,\Omega}\lesssim h^k|v|_{H^k(\Omega)}.
\end{align*}
Similarly, let $\bP_h:\bM\to\bM_h$ be the vector-valued $L^2$-projection. Let $\bF_h:\bM\to \bM_h$ be the classical Fortin interpolation operator satisfying the following commutative property (see, e.g., \cite{BBF13}): For all $\bpsi\in\bM$,
\begin{align*}
\div \bF_h\psi = P_h (\div  \psi).
\end{align*}
This provides the orthogonality property
\begin{align}
(\div \bpsi-\div  \bF_h\bpsi,v_h)_{0,\Omega} = 0 \qquad\forall v_h\in U_h.\label{Fortin}
\end{align}
Next, for all $\btau\in [H^{k+1}(\Omega)]^d$ with $\div \btau \in H^{k+1}(\Omega)$, we have the following properties (see, e.g., \cite{BBF13}) 
\begin{subequations}
\begin{align}
\|\btau-\bP_h\btau\|_{0,\Omega} + \|\btau-\bF_h\btau\|_{0,\Omega} &\lesssim h^{k+1} |\btau|_{[H^{k+1}(\Omega]^d},\\
\|\div (\btau-\bP_h\btau)\|_{0,\Omega} + \|\div (\btau-\bF_h\btau)\|_{0,\Omega} &\lesssim h^{k+1} |\div \btau|_{H^{k+1}(\Omega)}.
\end{align}
\end{subequations}
 
We define the errors as
\begin{align*}
 u-u_h  &= (u-P_hu) + (P_hu-u_h) =: \rho_u +\theta_u,\\
 \bsigma-\bsigma_h &= (\bsigma - \bP_h\bsigma) + (\bP_h\bsigma-\bsigma_h) =:\brho_{\bsigma} + \theta_{\bsigma},\\
 \bphi-\bphi_h &= (\bphi -\bF_h\bphi) + (\bF_h\bphi - \bphi_h) =:\rho_{\bphi} + \theta_{\bphi}.
\end{align*}
 
\begin{theorem}
The following error estimate holds:
\begin{align*}
&\|\theta_u\|^2_{L^{\infty}(\ol{J};L^2(\Omega))}+ \gamma\|\div \theta_{\bsigma}\|_{L^2(J;L^2(\Omega))}^2 + \|\theta_{\bsigma}\|^2_{L^{2}(;,L^2(\Omega))} \nonumber\\&\qquad\lesssim\|\partial_t\rho_u\|^2_{L^{2}(J;L^2(\Omega))}+\|\rho_u\|_{L^{2}(J;L^2(\Omega))}^2+\gamma\|\div \rho_{\bsigma}\|_{L^{2}(J;L^2(\Omega))}^2\\
&\quad \qquad +\|\rho_{\bsigma}\|^2_{L^{2}(J;L^2(\Omega))}+\|\rho_{\bphi}\|^2_{L^{2}(J;L^2(\Omega))}.
\end{align*}
\end{theorem}
\begin{proof}
Testing the continuous problem \eqref{saddle-cont} against $(v,\btau):= (v_h,\btau_h)$ and subtracting \eqref{saddle-dis} from \eqref{saddle-cont}, we get the following error equations
\begin{subequations}\label{saddle-error}
    \begin{align}
    (\partial_t(u-u_h), v_h)_{0,\Omega} +\gamma (a((u-u_h,\bsigma-\bsigma_h),(v_h,\btau_h)) & \nonumber \\+ b((v_h,\btau_h),\bphi-\bphi_h))+ c(\bsigma-\bsigma_h,\btau_h)+(g(u)-g(u_h),v_h)_{0,\Omega}&{}={}  0, \\
    b((u-u_h,\bsigma-\bsigma_h),\bpsi_h) &{}={}  0, \label{error-b}
    \end{align}
\end{subequations}
for all $(v_h,\btau_h) \in \bbV_h$ and all $\bpsi_h \in  \bM_h$. Since  $\div \theta_{\phi}\in U_h$, the $L^2$-orthogonality \eqref{L2} implies that
\begin{align*}
b((u-P_hu,\bsigma-\bP_h\bsigma),\bpsi_h) = (\bsigma-\bP_h\bsigma,\bpsi_h)_{0,\Omega} + (u-P_hu,\div \bpsi_h)_{0,\Omega} = 0,
\end{align*} 
for all $\bpsi_h\in\bM_h$. This in \eqref{error-b} provides that $b((\theta_u,\theta_{\bsigma}),\bpsi_h) = 0$ for all $\bpsi_h\in \bM_h$. Invoking this, using the error decompositions, and choosing $(v_h,\btau_h) = (\theta_u,\theta_{\bsigma})$ and $\bpsi_h=\theta_{\phi}$, we arrive at
\begin{align}
& \frac12\frac{d}{dt}\|\theta_u\|^2_{0,\Omega}+ \gamma\|\div \theta_{\bsigma}\|_\Omega^2 + \|\theta_{\bsigma}\|^2_{0,\Omega} = -(\partial_t\rho_u,\theta_u)_{0,\Omega} - \gamma ((\div  \rho_{\bsigma},\div  \theta_{\bsigma})_{0,\Omega}\nonumber \\
&\qquad +b((\theta_u,\theta_{\bsigma}),\rho_{\bphi}))  - (\rho_{\bsigma},\theta_{\bsigma})_{0,\Omega}- (g(u)-g(u_h),\theta_u)_{0,\Omega}.\label{32}
\end{align}
The orthogonalities \eqref{L2} and \eqref{Fortin} lead to $(\partial_t\rho_u,\theta_u)_{0,\Omega}=0$, and $b((\theta_u,\theta_{\bsigma}),\rho_{\bphi})= (\theta_{\bsigma},\rho_{\bphi})_{0,\Omega}$. 
The definition of the nonlinear term $g$ and the Young inequality show that
\begin{align*}
\|g(u)-g(u_h)\|_{0,\Omega} &\leq \|u-u_h\|_{0,\Omega}\|u^2+uu_h+u_h^2-1\|_{\infty,\Omega}\\
& \leq C_L \|u-u_h\|_{0,\Omega}\leq C_L (\|\rho_u\|_{0,\Omega}+\|\theta_u\|_{0,\Omega}),
\end{align*}
where we used the Sobolev embedding result $H^2(\Omega) \hookrightarrow L^\infty(\Omega) $ and the regularity estimate for both $u$ and $u_h$ in the second step, and the triangle inequality in the last step. The previous displayed estimate and the Young inequality in \eqref{32} prove that
\begin{align*}
&\frac12\frac{d}{dt}\|\theta_u\|^2_{0,\Omega}+ \frac{\gamma}{2}\|\div \theta_{\bsigma}\|_\Omega^2 + \frac12\|\theta_{\bsigma}\|^2_{0,\Omega} \\&\leq\frac12\|\partial_t\rho_u\|^2_{0,\Omega}+\frac{\gamma}{2}\|\div \rho_{\bsigma}\|_{0,\Omega}^2+\|\rho_{\bphi}\|^2_{0,\Omega}+\|\rho_{\bsigma}\|^2_{0,\Omega}+C_L\|\rho_u\|_{0,\Omega}^2+C\|\theta_u\|^2_{0,\Omega}.
\end{align*}
Integrating from $0$ to $t$ and invoking Gronwall lemma, we conclude the proof.
\end{proof}

\section{Numerical results}\label{sec:results} 
This section presents a series of numerical experiments illustrating the performance of the proposed mixed formulation for the extended Fisher--Kolmo\-go\-rov problem \eqref{def:time_dep_model_problem}. All tests are carried out  using exact solutions in closed form so that the corresponding forcing term and boundary conditions can be manufactured. We consider two different types of boundary conditions: the simply supported case and the Cahn--Hilliard case.

The spatial discretisation employs Raviart--Thomas spaces of order $k=0,1$ for the flux variables $\bsigma$ and $\bphi$, coupled with piecewise polynomial approximations of compatible degree for the remaining unknown $u$ in the mixed system. We consider a sequence of uniformly refined meshes of $\Omega$ with mesh-size $h$. Errors are reported in the natural $H(\mathrm{div})$-norm associated with the mixed formulation, and, in addition, in the $L^2$-norm for the biharmonic solution $u$, all at the final time step. The discretisation in time is done with backward Euler's scheme, utilising a small constant time step so that the total error is dominated by the spatial error. 

Let $e(h)$ denote the computed error associated with the solution $u$, the gradient 
$\bsigma$ or the Lagrange multiplier $\bphi$ on a mesh with size $h$. The experimental order of convergence (EOC) in space between two meshes of sizes $h_1$ and $h_2$ ($h_2<h_1$) is defined by
\[
\texttt{EOC}(h_1,h_2)
\;=\;
\frac{\log\big(e(h_1)/e(h_2)\big)}{\log(h_1/h_2)}.
\]
We also recall  that when reporting EOCs in the tables, we follow the convention that a rate of $r$ indicates the error behaves like
\[
e(h)=\mathcal{O}(h^{r}).
\]

For the case of Cahn--Hilliard boundary conditions, the zero-mean constraint for $u$ is imposed with a real Lagrange multiplier requiring to add two terms in the left-hand side system.

\begin{table}[t!]
    \caption{Convergence for the Fisher--Kolmogorov equation at the final time $T = 0.1$,  against manufactured solutions for different polynomial degrees in the 2D case with simply supported boundary conditions. Here we set $\gamma = 1$. The symbol $\star$ here, and in the tables below, indicates that no convergence rate has been computed for the coarsest mesh refinement.}
    \label{tab:cv-ss}
    \centering
    \begin{tabular}{|rc|cccccc|}
    \hline
    \texttt{DoF} & $h$ & $e(u)$ & \texttt{EOC} & $e(\bsigma)$ & \texttt{EOC} & $e(\bphi)$ & \texttt{EOC} \\
               \hline 
               \hline
               \multicolumn{8}{|c|}{Ultra-weak scheme with $k=0$}\\
               \hline
    40 & 0.7071 & 2.72e-02 & $\star$  & 5.42e-01 & $\star$  & 1.07e+01 & $\star$  \\ 
   144 & 0.3536 & 1.42e-02 & 0.936 & 2.85e-01 & 0.927 & 5.61e+00 & 0.936 \\
   544 & 0.1768 & 7.18e-03 & 0.987 & 1.44e-01 & 0.982 & 2.84e+00 & 0.981 \\
  2112 & 0.0884 & 3.60e-03 & 0.997 & 7.23e-02 & 0.996 & 1.43e+00 & 0.995 \\
  8320 & 0.0442 & 1.80e-03 & 0.999 & 3.62e-02 & 0.999 & 7.14e-01 & 0.999 \\
 33024 & 0.0221 & 9.00e-04 & 1.000 & 1.81e-02 & 1.000 & 3.57e-01 & 1.000 \\
\hline 
               \multicolumn{8}{|c|}{Ultra-weak scheme with $k=1$}\\
               \hline
   120 & 0.7071 & 8.19e-03 & $\star$ & 1.62e-01 & $\star$ & 3.19e+00 & $\star$ \\ 
   448 & 0.3536 & 2.15e-03 & 1.931 & 4.28e-02 & 1.919 & 8.44e-01 & 1.916 \\
  1728 & 0.1768 & 5.45e-04 & 1.979 & 1.09e-02 & 1.978 & 2.14e-01 & 1.977 \\
  6784 & 0.0884 & 1.37e-04 & 1.995 & 2.73e-03 & 1.994 & 5.38e-02 & 1.994 \\
 26880 & 0.0442 & 3.42e-05 & 1.999 & 6.82e-04 & 1.999 & 1.35e-02 & 2.000 \\
107008 & 0.0221 & 8.63e-06 & 1.986 & 1.72e-04 & 1.986 & 3.01e-03 & 2.013  \\             
     \hline
    \end{tabular}

\end{table}

The first set of tests corresponds to the simply supported case and on the unit square domain $\Omega = (0,1)^2$ and the time domain $[0,0.1]$ with the time step $\Delta t= 0.01$. The manufactured solution 
\[ u = t \sin(\pi x) \sin(\pi y)\] 
is chosen so that the solution is sufficiently smooth, ensuring that the theoretical convergence rates are attainable, and it also satisfies the simply supported boundary conditions. Note that in this case we take the non-homogeneous version of the extended Fisher--Kolmogorov equation and the corresponding source term is computed from the manufactured solutions. Table~\ref{tab:cv-ss} displays the error history (individual errors and estimated rates of convergence) for RT$_0$ and RT$_1$, respectively. The results confirm the theoretical error estimates. For RT$_0$ we observe first-order convergence in the energy norm while  for RT$_1$, the convergence rates increase to second order. In all cases the asymptotic regime is reached rapidly, and the computed orders of convergence match the approximation properties of the finite element spaces.

The second set of experiments considers the Cahn--Hilliard boundary conditions, also in the unit square domain. The   manufactured solution to the extended Fisher--Kolmogorov problem is now
\[ u = t\cos(\pi x) \cos(\pi y),\] 
and it is used to prescribe the data as well as the exact mixed variables. The corresponding error history is reported in Table~\ref{tab:cv-CH}, and sample solutions for the scheme with $k=1$ and on a fine mesh are depicted at time $T=0.1$ in Figure~\ref{fig:cv-2d}. 

\begin{table}[t!]
    \caption{Convergence for the extended Fisher--Kolmogorov equation at the final time $T = 0.1$, against manufactured solutions for different polynomial degrees in the 2D case with Cahn--Hilliard boundary conditions. Here we set $\gamma = 1$.}
    \label{tab:cv-CH}
    \centering
    \begin{tabular}{|rc|cccccc|}
    \hline
    \texttt{DoF} & $h$ & $e(u)$ & \texttt{EOC} & $e(\bsigma)$ & \texttt{EOC} & $e(\bphi)$ & \texttt{EOC} \\
               \hline 
               \hline 
               \multicolumn{8}{|c|}{Ultra-weak scheme with $k=0$}\\
               \hline
           41 & 0.7071 & 2.72e-02 & $\star$ & 5.45e-01 & $\star$ & 1.08e+01 & $\star$ \\ 
   145 & 0.3536 & 1.43e-02 & 0.928 & 2.86e-01 & 0.931 & 5.62e+00 & 0.938 \\
   545 & 0.1768 & 7.19e-03 & 0.991 & 1.44e-01 & 0.985 & 2.85e+00 & 0.981 \\
  2113 & 0.0884 & 3.60e-03 & 0.998 & 7.23e-02 & 0.996 & 1.43e+00 & 0.995 \\
  8321 & 0.0442 & 1.80e-03 & 1.000 & 3.62e-02 & 0.999 & 7.14e-01 & 0.999 \\
 33025 & 0.0221 & 9.00e-04 & 1.000 & 1.81e-02 & 1.000 & 3.57e-01 & 1.000 \\
 \hline 
               \multicolumn{8}{|c|}{Ultra-weak scheme with $k=1$}\\
               \hline
   121 & 0.7071 & 8.24e-03 & $\star$ & 1.63e-01 & $\star$ & 3.21e+00 & $\star$ \\ 
   449 & 0.3536 & 2.15e-03 & 1.937 & 4.29e-02 & 1.928 & 8.46e-01 & 1.923 \\
  1729 & 0.1768 & 5.45e-04 & 1.982 & 1.09e-02 & 1.981 & 2.15e-01 & 1.980 \\
  6785 & 0.0884 & 1.37e-04 & 1.995 & 2.73e-03 & 1.995 & 5.38e-02 & 1.995 \\
 26881 & 0.0442 & 3.42e-05 & 1.999 & 6.82e-04 & 1.999 & 1.33e-02 & 2.020 \\
107009 & 0.0221 & 7.37e-06 & 1.997 & 1.76e-04 & 1.995 & 3.02e-03 & 2.014 \\
               \hline
    \end{tabular}
\end{table}

\begin{figure}[t!]
    \centering
    \includegraphics[width=0.325\linewidth]{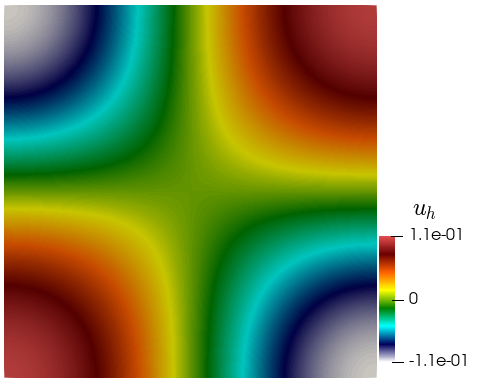}
\includegraphics[width=0.325\linewidth]{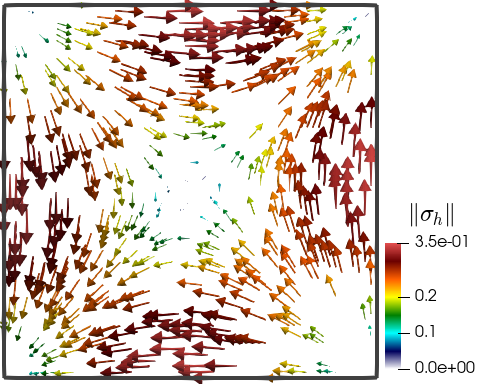}
\includegraphics[width=0.325\linewidth]{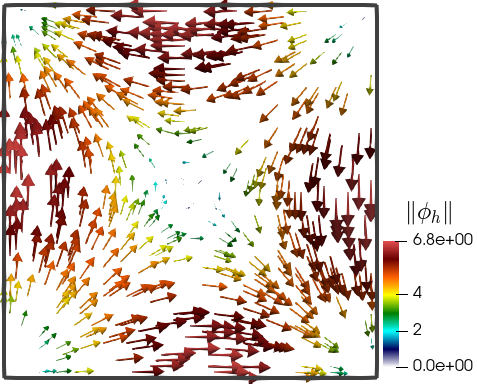}
    \caption{Approximate solutions for the Fisher--Kolmogorov  problem at time $T=0.1$}, in 2D using Cahn--Hilliard boundary conditions. The results correspond to the second-order mixed scheme.
    \label{fig:cv-2d}
\end{figure}

Next, we also show spatial convergence in the 3D case, taking the unit cube domain $\Omega = (0,1)^3$ and considering the manufactured solution 
\[ u = t\sin(\pi x) \sin(\pi y)\sin(\pi z).\] 
For conciseness we only show the case of simply supported boundary conditions, and take the time interval [0,0.01],  $\Delta t = 0.001$, and $\gamma =0.01$. The convergence is reported in Table~\ref{tab:cv-3d}, while we also show samples of approximate primal and mixed variables in Figure~\ref{fig:cv-3d}. 

\begin{table}[t!]
    \caption{Convergence against manufactured solutions in 3D  and using simply supported boundary conditions. Errors computed at the final time $T = 0.01$.}
    \label{tab:cv-3d}
    \centering
   \begin{tabular}{|rc|cccccc|}
    \hline
    \texttt{DoF} & $h$ & $e(u)$ & \texttt{EOC} & $e(\bsigma)$ & \texttt{EOC} & $e(\bphi)$ & \texttt{EOC} \\
               \hline 
               \hline 
                \multicolumn{8}{|c|}{Ultra-weak scheme with $k=0$}\\
               \hline
     42 & 1.7321 & 2.90e-03 & $\star$ & 8.26e-02 & $\star$ & 3.15e+00 & $\star$ \\ 
   288 & 0.8660 & 1.80e-03 & 0.693 & 5.38e-02 & 0.619 & 1.65e+00 & 0.938 \\
  2112 & 0.4330 & 9.60e-04 & 0.904 & 2.88e-02 & 0.901 & 8.55e-01 & 0.945 \\
 16128 & 0.2165 & 4.88e-04 & 0.975 & 1.47e-02 & 0.974 & 4.34e-01 & 0.977 \\
125952 & 0.1083 & 2.45e-04 & 0.994 & 7.36e-03 & 0.994 & 2.18e-01 & 0.994 \\
995328 & 0.0541 & 1.23e-04 & 0.998 & 3.69e-03 & 0.998 & 1.09e-01 & 0.998\\
 \hline 
               \multicolumn{8}{|c|}{Ultra-weak scheme with $k=1$}\\
               \hline
   168 & 1.7321 & 2.08e-03 & $\star$ & 6.07e-02 & $\star$ & 2.15e+00 & $\star$  \\
  1200 & 0.8660 & 6.36e-04 & 1.708 & 1.89e-02 & 1.682 & 5.70e-01 & 1.913 \\
  9024 & 0.4330 & 1.73e-04 & 1.880 & 5.16e-03 & 1.872 & 1.53e-01 & 1.898 \\
 69888 & 0.2165 & 4.42e-05 & 1.967 & 1.32e-03 & 1.966 & 3.91e-02 & 1.968 \\
549888 & 0.1083 & 1.11e-05 & 1.992 & 3.32e-04 & 1.991 & 9.84e-03 & 1.991 \\
4362240 & 0.0541 & 2.78e-06 & 1.998 & 8.32e-05 & 1.998 & 2.46e-03 & 1.998  \\      
               \hline
    \end{tabular}

\end{table}

\begin{figure}[t!]
    \centering
    \includegraphics[width=0.325\linewidth]{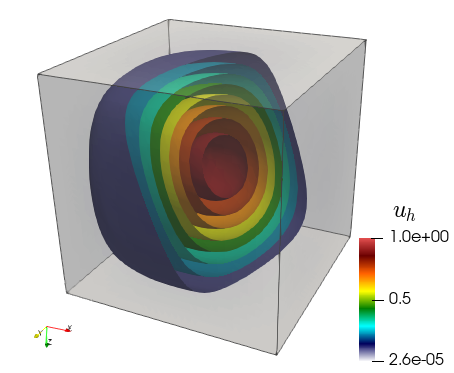}
\includegraphics[width=0.325\linewidth]{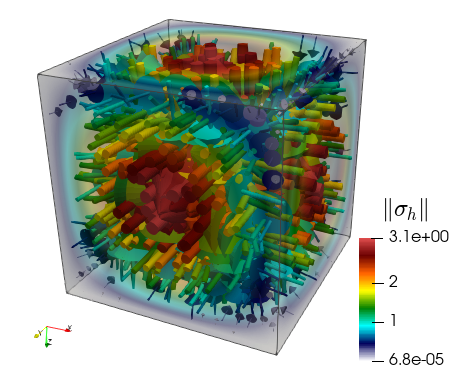}
\includegraphics[width=0.325\linewidth]{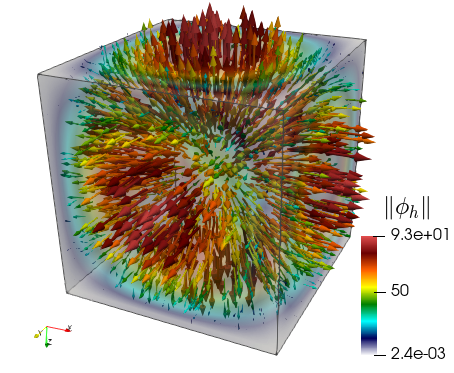}
    \caption{Approximate solutions for the biharmonic problem in 3D using simply supported boundary conditions. The results correspond to the lowest-order mixed scheme.}
    \label{fig:cv-3d}
\end{figure}

The observed convergence rates are  in full agreement with the theoretical predictions from Theorem~\ref{th:cv}. For both polynomial degrees, the optimal error decay of $O(h^{k+1})$ is obtained in all norms. This demonstrates that the proposed mixed formulation is robust with respect to the type of boundary conditions imposed and works in both 2D and 3D.

\begin{figure}[t!]
    \centering
    \includegraphics[width=0.98\textwidth]{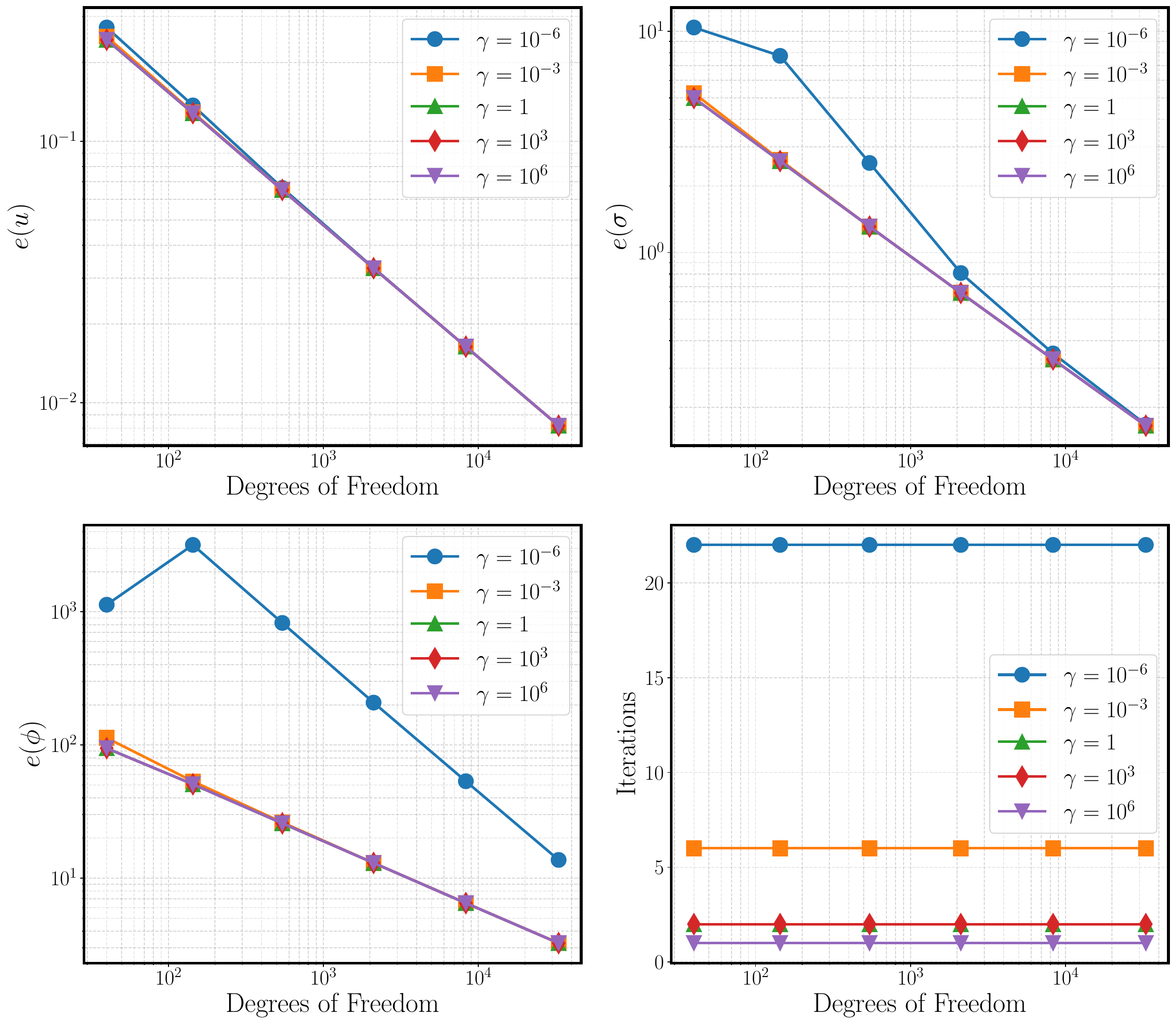}
    \caption{Performance (error history and number of Newton iterations to converge) of the mixed finite element scheme for the extended Fisher--Kolmogorov problem, taking different values of the model parameter $\gamma$.}
    \label{fig:parametric}
\end{figure}

It is also important to address the dependence of the formulation on the values of the parameter $\gamma$. We perform a simple parametric study, taking as base-line case simply supported boundary conditions, the unit square domain, the time domain $[0,1]$ and $\Delta t = 1$ (that is, we only do one time step). The plot in Figure \ref{fig:parametric} shows the  error decay for each unknown and number of required Newton iterations for the convergence of the finite element scheme for different values of $\gamma$. We observe that the convergence of the primal unknown is not affected by variations in $\gamma$ that optimal convergence is attained except for the case of $\gamma = 10^{-6}$ (where essentially the Jacobian matrix becomes of low rank since the off-diagonal blocks associated with the bilinear form $b$ vanish), and only for the Lagrange multiplier. We also observe that the lower the value of $\gamma$ the more difficult is for the nonlinear solver to converge.

Finally, we perform a simulation of the extended Fisher--Kolmogorov equation with simply supported boundary conditions on a 3D geometry of a gear. The time interval is [0, 0.5] and the constant time step is $\Delta t $ = 0.1. This can be considered as the 3D extension of the 2D gear domain convergence tests done for the same equation (in primal form and using generalised finite difference schemes) in \cite[Case 2]{ju2024analysis}. 

The initial condition is taken as the solution of the steady 
problem with source term $f(x,y,z) = 100 \sin(2 \pi x) \sin(2 \pi y) \sin(3 \pi z)$, and 
homogeneous simply supported boundary conditions. The simulation is 
performed up to the final time $T=0.5$ with constant time step $\Delta t = 0.1$. We use    
the parameter $\gamma = 1$ and take the second-order scheme with $k=1$, which give for this 
mesh resolution a total of 694280 degrees of freedom. The numerical solutions are displayed in Figure \ref{fig:gear}, showing smooth concentration of potential near the centre of the gear.
	
\begin{figure}[t!]
    \centering
   \includegraphics[width=0.325\linewidth]{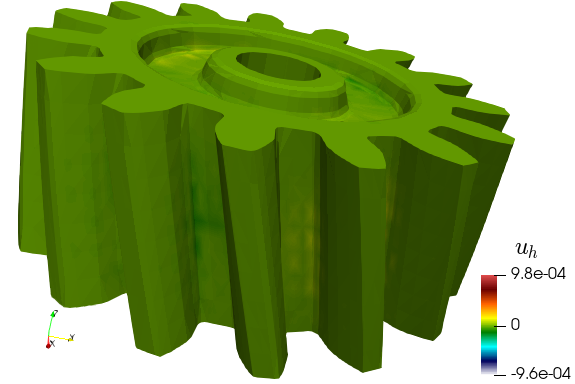}
    \includegraphics[width=0.325\linewidth]{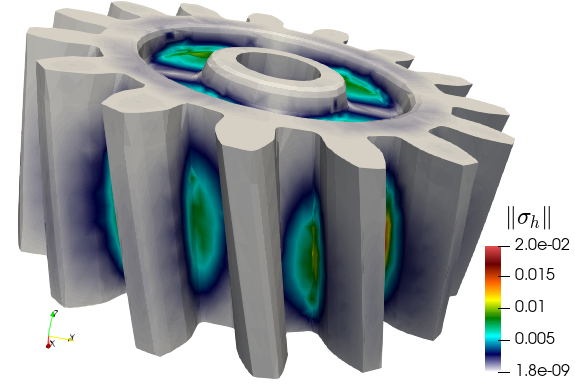}
    \includegraphics[width=0.325\linewidth]{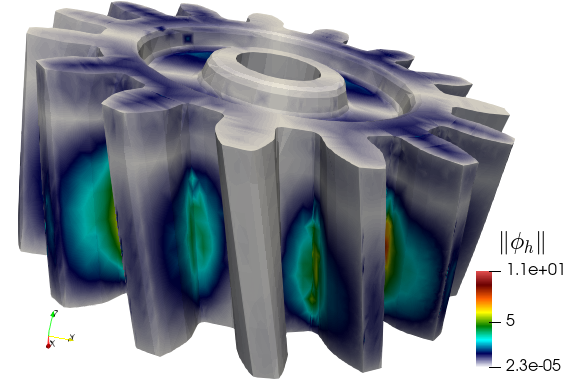}
    \caption{Numerical solution of the extended Fisher--Kolmogorov equation on a gear model at $T=0.5$. Potential, gradient, and Lagrange multiplier obtained with the second-order scheme.}
    \label{fig:gear}
\end{figure}	

\acknowledgement{The work was started when B. P. Lamichhane visited
Monash University in 2023. He gratefully acknowledges Monash University 
for their hospitality.
Part of this work was completed when 
all authors met in a MATRIX workshop at Creswick in May, 2025. We gratefully acknowledge the support from the MATRIX Institute. 
This work has been also supported by the Australian Research Council through the Future Fellowship grant FT220100496.} 

\bibliographystyle{spmpsci}
\bibliography{total}
\end{document}